\documentclass[a4paper,11pt,twoside]{amsart}

\usepackage{amsmath}
\usepackage{amsthm}
\usepackage{amssymb}
\usepackage[all]{xy}
\usepackage{hyperref}
\usepackage{rotating}
\usepackage{booktabs}

\newtheorem{theorem}{Theorem}[section]
\newtheorem{proposition}[theorem]{Proposition}
\newtheorem{lemma}[theorem]{Lemma}
\newtheorem{corollary}[theorem]{Corollary}

\numberwithin{equation}{section}

\theoremstyle{definition}
\newtheorem{definition}[theorem]{Definition}
\newtheorem{remark}[theorem]{Remark}

\newcommand{\supp}{\operatorname{Supp}}

\newcommand{\Db}{{\rm D}^{\rm b}}

\newcommand{\Pic}{{\rm Pic}}
\newcommand{\ch}{{\rm ch}}

\newcommand{\rk}{{\rm rk}}

\newcommand{\Hom}{{\rm Hom}}

\newcommand{\td}{{\rm td}}

\newcommand{\Hilb}{{\rm Hilb}}
\newcommand{\WIT}{{\text{-WIT}}}
\newcommand{\IT}{{\text{-IT}}}
\renewcommand{\dim}{{\rm dim}\,}
\renewcommand{\supp}{{\rm Supp}}

\newcommand{\coker}{{\rm coker}}

\newcommand{\dual}{^{\vee}}

\newcommand{\Ext}{{\rm Ext}}

\newcommand{\cal}{\mathcal}

\newcommand{\ke}{{\cal E}}

\newcommand{\kl}{{\cal L}}

\newcommand{\ko}{{\cal O}}
\newcommand{\kp}{{\cal P}}
\newcommand{\kq}{{\cal Q}}

\newcommand{\kt}{{\cal T}}

\newcommand{\ZZ}{\mathbb{Z}}

\newcommand{\CC}{\mathbb{C}}

\newcommand{\PP}{\mathbb{P}}

\def\ee{\varepsilon}

\begin{document}

\title{The Euclid-Fourier-Mukai algorithm for elliptic surfaces}
\author[M.\ Bernardara, G.\ Hein]{Marcello Bernardara and Georg Hein}

\address{Univerist\"at Duisburg-Essen, Fakult\"at f\"ur Mathematik. Universit\"atstrasse 2, 45117 Essen (Germany)}
\email{M.B.: marcello.bernardara@uni-due.de}

\email{G.H.: georg.hein@uni-due.de}

\begin{abstract}
We describe the birational correspondences, induced by the Fourier-Mukai
functor, between moduli spaces of semistable sheaves on elliptic
surfaces with sections, using the notion of $P$-stability in the derived
category. We give explicit conditions to determine whether these
correspondences are isomorphisms. This is indeed not true in general and
we describe the cases where the birational maps are Mukai flops.
Moreover, this construction provides examples of new compactifications
of the moduli spaces of vector bundles via sheaves with torsion and via
complexes. We finally get for any fixed dimension an isomorphism between the Picard groups of the
moduli spaces.
\end{abstract}

\maketitle

\section{Introduction}
Let $X \to \PP^1$ be a K3 elliptic surface with a section and
at most finitely many nodal
singular fibers and $\widehat{X} \to \PP^1$ its relative generalized Jacobian.  The relative
Poincar\'e bundle gives a Fourier-Mukai equivalence $\Psi: \Db(X) \to
\Db(\widehat{X})$. In this paper, we consider moduli spaces $M_X(v)$ of
semistable sheaves with a fixed Mukai vector $v$ on $X$ and we describe
how we can fix a Mukai vector $\hat{v}$ on $\widehat{X}$ in order to
describe via $\Psi$ a birational map $\varepsilon$ between $M_X(v)$ and
$M_{\widehat{X}}(\hat{v})$. Moreover, we give conditions on $v$ in order to
get an isomorphism or a very nice birational map, that is a Mukai flop.
These conditions depend only on the dimension of the moduli space.
Finally, we show that, with no restriction on the dimension of $M_X(v)$,
the birational maps induced by the Fourier--Mukai equivalence induce
an isomorphism between the Picard groups. It follows in particular
that for all $v$ such that $M_X(v)$ is smooth of dimension $2t$, the Picard
group of $M_X(v)$ is isomorphic to the Picard group of $\Hilb^t(X)$. 
Remark that we restrict, in order to have simple notations, to the
case where $X$ is K3. This restriction should not be necessary, because
it used throughout the paper only to simplify numerical computations.
The stated results should then hold for any elliptic surface, even
though the numerical data have to be changed. The existence of a section
of $X \to \PP^1$ gives an isomorphisms $X \cong \widehat X$. However, we think
it clarifies the exposition if we distinguish between both objects.

The key point needed for this result is the construction of a stability condition in the derived
category $\Db(X)$, called $P$-stability, defined in \cite{HePold,HP}, which recovers the slope stability and
is well behaved with respect to derived
equivalences. We will describe how to recover the $\mu$-stability with respect to a given polarization and
we will then look at semistable sheaves as stable objects in the derived category.
The interest of our results is then twofold: the study of birational maps between moduli spaces
and the construction
of moduli spaces of objects in the derived category.

Moduli spaces of semistable sheaves on elliptic surfaces have been intensively studied in the last twenty
years. The cornerstone in the theory is \cite{friedellip}. In that paper, the rank two case is studied, and most of the
basic notions are introduced. For example, the use of a fiberlike (called suitable there) polarization
(see also \cite{OG}), and the description of a birational map between moduli spaces. In the case of an elliptic surface
with a section $\sigma$, Friedman considered the moduli spaces of rank $2$ sheaves with odd fiber degree. Without loss
of generality, one can fix the Mukai vector $v$ corresponding to the Chern classes $(2,\sigma-tf,1)$,
where $t$ is a nonnegative integer and $f$ denotes the
fiber. Then Friedman shows that $M_X(v)$ is smooth, projective, $2t$ dimensional and birational to
$\Hilb^t(X)$. Moreover, for $t =1,2$ the birational map is actually an isomorphism. He then conjectures
(\cite[Part III, Conjecture 4.13]{friedellip}) that it is an isomorphism for all $t$.

The point of view of Fourier-Mukai transforms has been introduced into
this problem by Bridgeland \cite{bridgellip}. Consider $X$ and the
relative Jacobian $Y:=J_X(a,b)$, that is the elliptic surface whose
fibers are the moduli spaces of rank $a$, degree $b$ on the corresponding
elliptic fiber of $X$. Bridgeland constructs a derived equivalence
$\Psi: \Db(X) \to \Db(Y)$ as the Fourier-Mukai functor whose kernel is
the relative universal sheaf. He extends the result of Friedman
describing, for any Mukai vector $v$ with rank $r$ and fiber degree coprime,
a birational map between $M_X(v)$ and $\Hilb^t(Y)$, where $2t$ is the
dimension of $M_X(v)$ and $a$ and $b$ in the definition of $Y$ depend on
$v$. Moreover, he shows that these two spaces are isomorphic for small
values of $t$. Remark that this extends the result of Friedman (since $Y
\cong X$ if $X$ has a section) to arbitrary rank, but does not include
the case $t=r=2$.

In this paper, given $X$ and $v$,
we focus our attention to the Fourier-Mukai equivalence $\Psi: \Db(X) \to \Db(\widehat{X})$ (in Bridgeland's notations, $\widehat{X} =
J_X(1,0)$). We consider a Mukai vector $\hat{v}$ on $\widehat{X}$, uniquely determined by $v$ and the induced birational map
$$\varepsilon: M_X(v) \dashrightarrow M_{\widehat{X}}(\hat{v}).$$
We show that $\varepsilon$ is an isomorphism if the dimension $2t$ is smaller than
$2r$ and a Mukai flop if $r \leq t < r+d$. This Mukai flop turns out to be trivial if and only if $t=r=2$.
Finally, we show that, with no restriction on the dimension, the map $\ee$ induces an isomorphism
between the Picard groups of $M_X(v)$ and $M_{\widehat{X}}(\hat{v})$.

The choice of working with a very particular dual elliptic fibration is not restrictive. We construct indeed an algorithm
which recovers Bridgeland's
result by ours in a finite number of steps. Such algorithm is obtained combining the Euclidean algorithm applied to the
pair rank, fiber degree and the Fourier-Mukai transform and, starting with a $2t$-dimensional moduli space of semistable
sheaves whose rank and fiber degree are coprime, it ends with the Hilbert scheme $\Hilb^t(\widehat{X})$.
We call it the \it Euclid-Fourier-Mukai
algorithm. \rm Remark that a version of this algorithm for a single elliptic curve is used in Polishchuk's book
\cite{polishchuk} to recover the Atiyah's classification \cite{Ati} of indecomposable vector bundle on an
elliptic curve. In our case, by means of the Euclid-Fourier-Mukai algorithm, we can show that the bound given
by Bridgeland coincides with ours and is
indeed optimal for $r \geq 3$. Finally we argue how to find a counterexample to Friedman's conjecture
(take $t=4$ and $r=2$). Another important application of the Euclid--Fourier--Mukai algorithm gives
an isomorphism between $\Pic(M_X(v))$ and $\Pic(\Hilb^t(\widehat{X}))$ for all integers $t$.

When $\varepsilon$ is not an isomorphism, our construction allows to describe new compactifications for the moduli
spaces of vector bundles.
One of the good properties of $P$-stability is indeed that, once
we have a fine moduli space $M$ for $P$-stable objects and a triangulated equivalence $\Pi$, we have a fine moduli space
of $\Pi(P)$-stable objects which is isomorphic to $M$. In this paper, we define a $P$-stability
in $\Db(X)$ such that an object $E$ is $P$-stable if and only if it is a semistable torsion free sheaf with
Mukai vector $v$. The (smooth projective irreducible) moduli space of $P$-stable objects is $M_X(v)$,
and it is isomorphic to the moduli space
of $\Psi(P)$-stable objects
in $\Db(\widehat{X})$.

If $t\geq r$, the map $\varepsilon$ is not an isomorphism.
In particular, the moduli space $N_{\widehat{X}}$ of $\Psi(P)$-stable objects is
a smooth irreducible projective variety, isomorphic to $M_X(v)$, containing sheaves with torsion and sharing a dense open subset
with $M_{\widehat{X}}(\hat{v})$. On the other hand, consider the relative quasi-inverse functor $\Phi$
of $\Psi$ and the $P$-stability on $\Db(\widehat{X})$ corresponding to semistability of torsion free sheaves with Mukai vector
$\hat{v}$. Then $\Phi$ induces a birational inverse
of $\varepsilon$. For $t \geq r$, the moduli space $N_X$ of $\Phi(P)$-stable objects on $X$ is isomorphic to $M_{\widehat{X}}(\hat{v})$,
shares with $M_X(v)$ a dense open subset, and contains complexes of length 2. It follows that $N_X$ and $N_{\widehat{X}}$ provide
new compactifications of the moduli space of semistable vector bundles via complexes and torsion sheaves respectively.

The paper is organized as follows: in Section 2, given a smooth projective surface $X \to C$ fibering over a smooth
projective curve, and Chern classes $c_t=(r,c_1,c_2)$ with $r$ coprime to the fiber degree,
we describe the $P$-stability corresponding to the semistability of sheaves with Chern classes $c_t$. The first main ingredient
of this construction is the existence of a polarization $H$ on $X$ with respect to which the semistability of a torsion
free sheaf with Chern classes $c_t$ coincides with the semistability of its restriction to the generic fiber. The second main ingredient
is the description of semistability on a smooth projective curve via orthogonality with respect to some vector bundle.
In section 3, we consider an elliptic K3 surface $X$ and use the $P$-stability to describe an isomorphism between the moduli
spaces, provided that the dimension is smaller. We construct the Euclid-Fourier-Mukai algorithm and
compare our result with Bridgeland's one.
In section 4, we describe
the case where the birational map is a Mukai flop. In section 5, we prove that $\ee$ induces an isomorphism of the Picard groups.

We work exclusively over the complex field $\CC$.   

\subsection*{Acknowledgment}
This work has been supported by the SFB/TR 45
``Periods, moduli spaces and arithmetic of algebraic varieties''.
The authors thank Alina Marian and Dragos Oprea for writing \cite{marian-oprea},
which gave the motivation to study the isomorphism between the Picard groups, and for
their helpful and careful questions. It is a pleasure to thank David Ploog for comments
on an early version of this article.

\section{Postnikov stability data for fibered surfaces}\label{sect:Pstability}

Let $X$ be a smooth projective surface and $\pi: X \to C$ a fibration over a smooth projective curve.
In this section, we show that, once we fix Chern classes $c_t=(r,c_1,c_2)$ such that the fiber degree
of $c_1$ and $r$ are coprime, we can give a $P$-datum with respect to which an object in $\Db(X)$ is
stable if and only if it is a semistable torsion free sheaf with Chern classes $c_t$ and hence fixed Hilbert polynomial.
To do this, we observe that the restriction of a stable torsion free sheaf to the generic fiber has to be stable,
generalizing a result by O'Grady \cite{OG}.
We can even provide an upper bound on the number of fibers where the restriction of a given semistable sheaf can be
unstable. Moreover, semistability on a curve corresponds to orthogonality
to some vector bundle, as shown by Popa \cite{Popa}. The first part of the $P$-stability is then given by an orthogonality
condition with respect to
a finite number of vector bundles each one supported on a single fiber. The second part of the $P$-datum constrains the
object in the derived category to be a torsion free sheaf with fixed Hilbert polynomial.

\subsection{Postnikov stability}
We recall the definition of Postnikov
stability from \cite{HePold,HP}.

\begin{definition}
Let $\kt$ be a  $k$-linear triangulated category, $\{C_i \} _{i=-N \dots M}$ a finite collection
of objects from $\kt$, and $N_i^j$ natural
numbers with $j \in \ZZ$ and ${i=-N \dots M}$, with almost all
$N_i^j$ equal to zero.
We say that an object $E \in \kt$ is {\it Postnikov stable} with respect to
$(C_i,N_i^j)$ if
\begin{enumerate}
\item For all integers $i,j$ we have $\dim_k \Hom(C_i,E[j])=N_i^j$, and
\item there exists a convolution $T_0$ of the objects $C_0,\dots,C_M$
such that $\Hom(T_0,E[j])=0$ for all $j \in \ZZ$.
\end{enumerate}
We will shortly say $P$-datum instead of Postnikov datum.
\end{definition}
The main motivation for considering $P$-stability is that it is preserved by equivalences and
covers the classical semistability of
sheaves. As an example we have following
comparison theorem from \cite{HP}.

\begin{theorem} (\cite[Theorem 10]{HP})
Let $(X,\ko_X(1))$ be a smooth projective variety over $k$, and $p$ be a
Hilbert polynomial. Then there exists a P-datum $(C_i,N_i^j)$ in
$\Db(X)$ such that $E \in \Db(X)$ is P-stable if and only if
\begin{enumerate}
\item the complex $E$ is isomorphic to a sheaf $E_0$,
\item $E_0$ is a pure sheaf of Hilbert polynomial $p$,
\item $E_0$ is slope semistable with respect to $\ko_X(1)$.
\end{enumerate}
\end{theorem}

Once we know the existence of a fine moduli space for a given $P$-datum, an equivalence of triangulated categories
gives rise to another $P$-datum with isomorphic moduli space.

\begin{proposition}
Let $X$ and $Y$ be smooth projective varieties and $\Phi: \Db(X) \to \Db(Y)$ an equivalence. 
Consider an Hilbert polynomial $p(m)$ on $X$ such that the moduli space $M_X(p(m))$ of semistable sheaves
with Hilbert polynomial $p(m)$ on $X$ is fine. Fix on $X$ the $P$-datum corresponding to the Hilbert polynomial $p(m)$.
Then there exists a fine moduli space of $\Phi(P)$-stable objects on $Y$ isomorphic to $M_X(p(m))$.
\end{proposition}\label{prop:phi-is-iso-modspace}
\begin{proof}
It is clear that $E$ in $\Db(X)$ is $P$-stable if and only if $\Phi(E)$ is $\Phi(P)$-stable.
If $\ke$ is the universal sheaf on $M_X(p(m)) \times X$, then $\Phi(\ke)$ is the universal object for moduli on $Y$.
\end{proof}

\subsection{Generalizing a result of O'Grady}
Let $f$ denote the class of a smooth fiber of $\pi: X \to C$.

\begin{definition}\label{Def:discriminant}
Let $E$ be a vector bundle on $X$ of rank $r$
with Chern roots $\{\alpha_i\}_{i=1,\dots,r}$. The \it discriminant \rm $\Delta(E)$ of $E$
is defined as follows
\[ \Delta(E) := \sum_{1 \leq i < j \leq r} (\alpha_i-\alpha_j)^2 = 
(r-1)c_1^2(E)-2rc_2(E) \, .  \]
\end{definition}

We included the definition because many authors (as for example
\cite{huylehn}) use other scaling constants for $\Delta(E)$. For example
O'Grady includes in his definition of $\Delta$ the factor
$\frac{-1}{2r}$ (see Section IV. in \cite{OG}). With our definition we
have $\Delta(E) = \ch_2(E \otimes E\dual) =-c_2(E \otimes E\dual)$.
The Bogomolov inequality says that when $E$ is
semistable with respect to some polarization, then $\Delta(E) \leq 0$.

\begin{lemma}\label{lemma22}(Compare with Lemma IV.1 in \cite{OG})
Let $X$ be a smooth projective surface,
$0 \to E' \to E \to E'' \to 0$ be a short exact sequence of semistable
coherent sheaves of the same slope with respect to some polarization on
$X$. Let $A= r'c_1(E)-rc_1(E')$ where $r'$ and $r$ are the ranks of $E'$
and $E$ respectively. Then we have an inequality
\[ \frac{r^2}{4}\Delta(E) \leq A^2 \, . \]
\end{lemma}
\begin{proof}
The semistability of $E'$ and $E''$ yields
$c_2(E') \geq \frac{r'-1}{2r'}c_1^2(E')$, and 
$c_2(E'') \geq \frac{r-r'-1}{2(r-r')}c_1^2(E'')$.
The short exact sequence yields two equations
\[ c_1(E) = c_1(E')+c_1(E''), \quad \mbox{ and} \quad
c_2(E) = c_1(E')c_1(E'')+c_2(E') +c_2(E'') \,.  \]
The two inequalities give now
\[c_2(E) \geq  c_1(E')c_1(E'') + \frac{r'-1}{2r'}c_1^2(E') +
\frac{r-r'-1}{2(r-r')}c_1^2(E'') \, .\]
Now we express $c_1(E'')$ as $c_1(E)-c_1(E')$, and replace $c_1(E')$ by
$\frac{1}{r}(r'c_1(E)-A)$ to obtain
\[ 2rc_2(E) \geq (r-1)c_1^2(E)
- \frac{1}{r'(r-r')} A^2 \, .  \]
Having in mind that $\Delta(E) \leq  0$, and that $r'\in \{
1,2,\dots,r-1\}$ we obtain the stated inequality.
\end{proof}

\begin{lemma}\label{lemma23}
Let $(X,H)$ be a smooth polarized surface, $f$ and $A$ divisors on
$X$ such that $f^2=0$, $H.A=0$, and $H.f>0$. Then we have $A^2 \leq \displaystyle -\frac{(A.f)^2H^2}{(H.f)^2}$.
\end{lemma}
\begin{proof}
We consider the Hodge index theorem for the two line bundles $H$ and $(A+
\lambda f)$ and obtain $H^2(A+\lambda f)^2 \leq (H.A+\lambda H.f)^2$.
Thus, the quadratic equation $(H.f)^2 \lambda^2 -2 H^2 (A.f) \lambda
-H^2A^2 =0$ in $\lambda$ has at most one
rational solution. Considering the discriminant of this quadratic
equation, we obtain the stated inequality.
\end{proof}

Let us fix Chern classes  $c_t = (r,c_1,c_2)$ such that $f.c_1$ is coprime with $r$.
The following result is a generalization of O'Grady's result
\cite[Proposition I.1.6]{OG} and the proof follows almost literally the original one.

\begin{proposition}\label{prop24}
Let $E$ be a torsion free coherent sheaf $E$ with Chern classes $c_t$ and set $h:=H.f$,
and $N_0:= -\frac{r^2\Delta(E)h}{8}$. For any $N \geq N_0$ the following facts are equivalent:
\begin{enumerate}
\item $E$ is slope stable with respect to $H+N\cdot f$.
\item $E$ is slope semistable with respect to $H+N\cdot f$.
\item $E$ is slope stable with respect to $H+N_0\cdot f$.
\item $E$ is slope semistable with respect to $H+N_0\cdot f$.
\item The restriction of $E$ to a general fiber of $\pi$ is stable. 
\end{enumerate}
\end{proposition}
\begin{proof}
The main ingredient of the proof is the following Lemma.
\begin{lemma}\label{$*$}
For any $N' \geq N_0$, the sheaf $E$ cannot be strictly semistable
with respect to $H+N' \cdot f$.
\end{lemma}
\begin{proof}
Suppose $E$ is strictly semistable, that is there
exists a semistable subsheaf $E' \subset E$ of the same slope. Let $r'$
be the rank of $E'$, and set $A=r'c_1(E)-rc_1(E')$. We have by
Lemma \ref{lemma22} the inequality $\frac{r^2\Delta(E)}{4} \leq A^2$.
Applying Lemma
\ref{lemma23} to the polarization $H+N'\cdot f$ we obtain $A^2 \leq -
\frac{(H+N'\cdot f)^2}{h^2}$. Here we used that $r$ and $f.c_1(E)$ are
coprime which implies $A.f \ne 0$. Our choice of $N' \geq N_0$
yields $(H+N'\cdot f)^2 > 2 N_0 H.f$. We conclude that
$\frac{r^2\Delta(E)}{4} > -\frac{(H+N\cdot f)^2}{h^2}$.
This means that the above two inequalities for $A^2$ cannot be satisfied
simultaneously and proves the Lemma.
\end{proof}
\noindent Now
(1)$\implies$(2) and (3)$\implies$(4) are obvious, while
(2)$\implies$(1) and (4)$\implies$(3) follow immediately from Lemma \ref{$*$}.

(3)$\implies$(2). Suppose that $E$ is stable with respect to $H+N_0 \cdot
f$ and unstable with respect to $H+N\cdot f$. Then there exists a $N'
\in (N_0,N)$ such that $E$ is strictly semistable with respect to
$H+N'\cdot f$. This contradicts Lemma \ref{$*$}. We prove
(1)$\implies$(4) following the same pattern.

(1)$\implies$(5). We denote the fiber over $c \in C$ by $X_c$. If the
restriction $E_c:=E|_{X_c}$ is unstable for all $c \in C$,
there is a short exact sequence
$0 \to E' \to E \to E'' \to 0$ such that for almost all $c \in C$ 
we have  a destabilizing sub-bundle ${E'}_c \subset E_c$.
Setting $A:=\rk(E')c_1(E)-rc_1(E')$ we have therefore $f.A <0$.
Then for $\tilde N \gg N$ the subsheaf $E'$ is destabilizing with respect
to $H+\tilde N \cdot f$.
As before we deduce the existence of an $N' \in (N, \tilde N)$ such
that $E$ is strictly semistable with respect to $H+N'\cdot f$.
This contradicts Lemma \ref{$*$}. For the same reasons, (5)$\implies$(1).
\end{proof}

\begin{definition}\label{Def:fiberlike-polarization}
Fix Chern classes $c_t=(r, c_1, c_2)$ such that $c_1.f$ and $r$ are coprime.
A polarization $H$ on $X$ is \it almost $c_t$-fiberlike \rm if
for all torsion free coherent sheaves $E$ with Chern classes $c_t$ we have
\[ E \mbox{ is stable }\iff E \mbox{ is semistable }
\iff E|_{X_c} \mbox{ is stable for a  point } c \in C \, .\]
\end{definition}
The existence of an almost $c_t$-fiberlike polarization is shown in Proposition \ref{prop24}.

\subsection{The orthogonality condition}\label{subsect23}
We denote the genus of $f$ by $g$.
For any closed point $c \in C$, write $X_c$ for the fiber of
$\pi$ over $c$, and $\iota_c:X_c \to X$ for the embedding of $X_c$ into
$X$. Fix Chern classes $c_t$ and
an almost $c_t$-fiberlike polarization $H$ on $X$. We construct a family of vector bundles
such that any torsion free coherent sheaf of Chern classes $c_t$ is orthogonal
to some bundle of the family if and only if it is stable with respect to the
fixed polarization.

Recall first that the stability of coherent sheaves on smooth projective curves is
equivalent to the existence of orthogonal sheaves.
\begin{theorem}\label{thm25}(cf. \cite[Theorem 5.3]{Popa})
For a smooth projective curve $Y$ of genus $g$ and a coherent sheaf $E$
on $Y$
of rank $r$ and degree $d$ we have equivalent conditions
\begin{enumerate}
\item $E$ is a semistable vector bundle,
\item there exist a sheaf $F \ne 0$ such that $\Ext^i(F,E)=0$ for all
$i \in \ZZ$, and
\item for any line bundle $L$ of degree $r^2(\frac{d}{r}-(g-1))$ there exists a
vector bundle $F$ of rank $r^2$ and $\det(F) \cong L$ such that
$\Ext^i(F,E)=0$ for all $i \in \ZZ$.
\end{enumerate}
\end{theorem}

\begin{remark}\label{rem26}
If in the above Theorem the genus of $Y$ is one, then the sheaf $F$ in
(2) can be chosen of rank $r$ and degree $d$. This is a consequence of
Atiyah's classification \cite{Ati} of vector bundles on elliptic curves.
\end{remark}

\begin{definition}\label{Def:orthog-condition}
We define the \it orthogonality
condition \rm $(r,d)^\bot$ for objects $E \in \Db(X)$
of the bounded derived category of $X$ as follows:
\[ E \in \Db(X) \mbox{ is } (r,d)^\bot \stackrel{\mbox{\tiny Def}}{\iff}
\left\{
\begin{array}{l} \mbox{there exists a point } c \in C
\mbox{ and a coherent sheaf }\\
F \mbox{ on } X_c \mbox{ with } \rk(F)=r, \, \deg(F)=d
\mbox{ such that}\\
\Hom(\iota_{c*} F,E[j])=0 \mbox{ for all } j \in \ZZ .
\end{array} \right\}
\]
\end{definition}
For a coherent torsion free sheaf $E$ the orthogonality condition
$(r,d)^\bot$ is by Theorem \ref{thm25} equivalent to the stability of
$E$ restricted to a general fiber of $\pi$. by Proposition
\ref{prop24}, this is equivalent to the stability of $E$.
Resuming, we obtain the following Proposition.

\begin{proposition}\label{ortho}
For a torsion free sheaf $E$ with Chern classes $c_t$ and an almost $c_t$-fiberlike
polarization $H$, we have an equivalence
\begin{enumerate}
\item $E$ is $(r^2,r^2(\frac{d}{r}-(g-1)))^\bot$,
\item the restriction of $E$ to the general fiber of $\pi$ is stable,
\item $E$ is slope stable with respect to $H$, and
\item $E$ is slope semistable with respect to $H$.
\end{enumerate}
\end{proposition}

\subsection{A cohomological description of purity}
\begin{proposition}\label{purity}
Let $E$ be a coherent sheaf on $X$ with Chern classes $c_t$. Suppose that $E$ is $(r^2,r^2(\frac{d}{r}-(g-1)))^\bot$.
Then there exists a line bundle $L$ on $X$ such that
we have an equivalence
\[ \Hom(L,E) = 0 \iff E \mbox{ is torsion free.} \]
\end{proposition}
\begin{proof}
We take an integer $n$ such that
$n(H.f)<1-g+\frac{d}{r}+\frac{\Delta(E)}{2r}$ holds.
After fixing such an integer $n$ we choose $m \in \ZZ$ such that 
$\mu(\ko_X(nH+mf))>\mu(E)$ is satisfied. Set now $L=\ko_X(nH+mf)$.\\
``$\Longleftarrow$''
The line bundle $L$ is a line bundle of slope $\mu(L) > \mu(E)$.
Hence if $E$ is torsion free, then $E$ is stable by Proposition
\ref{ortho}. Whence we deduce $\Hom(L,E)=0$.\\
``$\Longrightarrow$'' From $\Hom(L,E) =0$ we can conclude that the
torsion sheaf $T(E)$ of $E$ is of pure dimension one. Since $E$
satisfies the orthogonality condition
$(r^2,r^2(\frac{d}{r}-(g-1)))^\bot$, we see that $T(E)$ is supported on
fibers of $\pi$. Thus we have $\ch(T(E)) = a[f]+b[{\rm pt}]$ in
$H^*(X,\ZZ)$ with $a \geq 0$. If $a=0$ we are done. So we suppose $a
\geq 1$.
Considering the short exact sequence
\[ 0 \to T(E) \to E \to \bar E \to 0 \]
we see that $\bar E$ also satisfies the orthogonality condition
$(r^2,r^2(\frac{d}{r}-(g-1)))^\bot$, is torsion free, and of discriminant
\[ \Delta(\bar E)  = \Delta(E) + 2(ad-br) \, .\]
Since $\bar E$ fulfills the orthogonality condition, it is stable for
the polarization $H+N \cdot F$ for $N \gg 0$ by Proposition
\ref{prop24}. So Bogomolov's inequality gives $\Delta(\bar E) \leq 0$.
We conclude $b \geq \frac{\Delta(E)+2ad}{2r}$. So
\[\mu(T(E)):= \frac{b}{a} \geq \frac{\Delta(E)}{2ar} + \frac{d}{r}  \geq
\frac{\Delta(E)}{2r}+  \frac{d}{r} \, .\]
The one dimensional sheaf $T(E) \otimes L\dual$ has Euler characteristic
$\chi(T(E) \otimes L\dual) = a(-L.F + \mu(T(E)) + 1-g)$.
The above choice of $L$ and the inequality for $\mu(T(E))$ imply that
$\chi(T(E) \otimes L\dual) > 0$ for $a>0$. Thus, we have $a=0$ iff
$\Hom(L,T(E)) =H^0(T(E) \otimes L\dual)=0$. Since $\Hom(L,E)=0$ implies
$\Hom(L,T(E)) =0$, we are done.
\end{proof}

\subsection{A $P$-datum for fibered surfaces}
If $E$ is a stable sheaf of given Hilbert polynomial, then there exists
an integer $N_0$ such that $H^i(E(N))=0$ for all $i >0$ and $N \geq
N_0$. The number $N_0$ depends only on the numerical invariants of $E$.
In \cite[Proposition 5]{HP}  it is shown that for any Hilbert polynomial
$p$ there exists a finite number of vector bundles $\{ L_i
\}_{i=0,..,M}$ and natural numbers $N_i$ such that we have an
implication for all objects $E \in \Db(X)$

\[ \left( \hom(L_i,E[j]) = \left\{ \begin{array}{ll}
N_i & \mbox{ for } j=0\\0& \mbox{ otherwise} \end{array}\right. \right)
\implies E  \mbox{ is a coherent sheaf of Hilbert polynomial } p \, .
\]
Using this and Propositions \ref{ortho} and \ref{purity}, we can sum up
the results of this section as follows:
\begin{theorem}\label{theorem:stable-is-Pstable}
Let $\pi:X \to C$ be a fibration of a smooth projective surface over a
curve, $f$ a general fiber of $\pi$,
$g$ the genus of $f$,
$c_t=(r,c_1,c_2)$ Chern classes, with $r$ coprime to $d:=f.c_1$.
Then there exists a polarization $H$ on $X$, objects $L,L_i \in
\Db(X)$, and integers $N_i$ such that:
an object $E$ satisfies the conditions 
\begin{enumerate}
\item[(h)] $\hom(L_i,E[j]) = \left\{ \begin{array}{ll}
N_i & \mbox{ for } j=0\\0& \mbox{ otherwise,} \end{array}\right.$
\item[(tf)] $\Hom(L,E)=0$, and 
\item[(o)] $E$ is $(r^2,r^2(\frac{d}{r}-(g-1)))^\bot$ 
\end{enumerate}
if and only if $E$ is a semistable sheaf of Chern classes $c_t$.
\end{theorem}

\begin{remark}
For an elliptic fibration $\pi:X \to C$  we can replace by Remark \ref{rem26} the
orthogonality condition (o) in the above theorem by $(r,d)^\bot$.
\end{remark}

\begin{remark}
The choice of a fiber and of a line bundle in condition (o) depends a priori on $E$. In order to
describe a $P$-datum, we need a finite number of objects in $\Db(X)$. Using the same arguments as in the
proof of \cite[Theorem 4]{HePold}, we can obtain them. It will be nevertheless more comfortable to
use in all the paper the condition (o), because of its explicit meaning. 
\end{remark}

\subsection{Exceptional fibers}
In this final subsection, we analyze, given a semistable sheaf $E$ of fixed
Chern classes, how unstable its restriction to a fiber can be. Here we
assume that the canonical class $K_X$ is trivial.

\begin{proposition}\label{ppp-fib}
Let $K_X$ be trivial, and $E$ be a vector bundle on $X$ with Chern classes
$c_t=(r,c_1,c_2)$. We assume that the rank $r$ and the fiber degree $c_1.f$
are coprime, and $E$ is stable for an almost $c_t$-fiberlike
polarization $H$. We set $t:=\frac{1}{2} h^1(\ke nd(E))= 1-\frac{1}{2}
\chi(\ke nd(E))$.
\begin{enumerate}
\item
Take a fiber $X_0$, denote $E_0:=E|_{X_0}$, and
consider a surjection $E_0 \stackrel{\alpha}{\to} Q$ with $Q$ a sheaf on
$X_0$ of positive rank $r_Q$. We write $E'_0$ for $\ker(\alpha)$ which
is a vector bundle on $X_0$ and $r'$ for its rank (remark that $r'=r-r_q$).
Then we have the inequality
\[ r r' ( \mu(E'_0)-\mu(E_0)) =  r r_Q\left( \mu(E_0)-\mu(Q) \right) \leq t \, .\]
\item
If the restriction $E_i$ of $E$ to the fibers $X_i$ is not
semistable, then we have short exact sequences $0 \to E'_i \to E_i \to
E''_i \to 0$ of $X_i$-vector bundles. Denote by $r'_i$ and
$r''_i$ the ranks of $E'_i$ and $E''_i$ respectively. We have an inequality
\[ \sum_{i=1}^l r r_i' (\mu(E_i')-\mu(E_i)) = \sum_{i=1}^l r
r_i'' (\mu(E_i)-\mu(E_i'')) \leq t \, . \]
\item The number of fibers $X_i$ where the restriction $E_i$ of $E$ to
$X_i$ is not stable cannot exceed $t$.
\end{enumerate}
\end{proposition}

\begin{proof}
(1) We consider the composition of surjections $E \to E_0 \to Q$ and
denote it kernel by $E'$. On a general fiber, the restriction of $E'$
coincides with the restriction of $E$. So we know that $E'$ is stable for
some almost-fiberlike polarization. This implies $h^0(\ke nd(E')) =
h^2(\ke nd(E'))=1$.
Thus, we have
\[\chi(\ke nd(E'))-\chi(\ke nd(E)) = 2t -h^1(\ke nd(E')) \leq 2t \,.\]
Using the Riemann-Roch theorem to compute $\chi(\ke nd(E'))-\chi(\ke
nd(E))$ we obtain
\[\chi(\ke nd(E'))-\chi(\ke nd(E))=(\ch(\ke nd(E'))-\ch(\ke nd(E)))
\td(X) = 2r_Q \deg(E_0) -2 r \deg(Q) \, . \]
Together with the above inequality this yields the inequality in the
statement (1). To check the equality $r r' ( \mu(E'_0)-\mu(E_0)) =  r
r_Q\left( \mu(E_0)-\mu(Q) \right)$ is straightforward.\\
(2) Is shown analogously to (1) by considering
$E'= \ker\left( E \to \bigoplus _{i=1} ^l E''_i \right)$.\\
(3) Follows from (2) since for destabilizing $E_i' \subset E_i$,
each individual summand in the above sum is a positive integer.
\end{proof}

\section{Isomorphisms between moduli spaces}\label{section:elliptic-fibration-and-fm}

From now on, we will focus our attention on elliptic surfaces.
In this section, we consider a K3 elliptic surface $X$, its relative Jacobian $\widehat{X}$ and the
Fourier-Mukai equivalence $\Psi$ between $\Db(X)$ and $\Db(\widehat{X})$ whose kernel is the relative
Poincar\'e bundle. Given a Mukai vector $v$ such that the moduli space $M_X(v)$ is smooth
projective irreducible and non-empty, we describe under which conditions $\Psi$ induces an isomorphism
between $M_X(v)$ and $M_{\widehat{X}}(\hat{v})$, where $\hat{v}$ depends only on $v$. In order to do this,
we consider the $P$-stability datum on $\Db(X)$ corresponding to the semistability of torsion
free sheaves with Mukai vector $v$. The choice of a very special Fourier-Mukai here is not restrictive.
We define the Euclid-Fourier-Mukai algorithm and
show that the construction described by Bridgeland in \cite{bridgellip} can be
recovered by ours in a finite number of step, see Lemma \ref{us=bridg}.

Let $\pi: X \to \PP^1$ be a relatively minimal elliptic surface with at most finitely many singular fibers.
Suppose moreover that singular fibers are at most nodal.
We restrict our attention to the K3 case in order to make all numerical calculation more clear.
However, the arguments can be extended to a general elliptic fibration with at most finitely many nodal
fibers over a smooth projective curve, provided that the numerical data are properly adapted.  
Let $f$ denote the class of a general fiber of $\pi$.
If $E$ is a torsion free sheaf on $X$, we call the integer $c_1(E).f$ its {\it fiber degree}.
If $E$ is an object in $\Db(X)$, then we define the fiber degree (resp. the rank)
of $E$ by the alternate sum of the fiber degrees (resp. the ranks) of its cohomology sheaves.

\subsection{Isomorphisms between moduli spaces} 
Let us fix a Mukai vector $v = (r,\Lambda,c)$, such that $\Lambda.f$ and $r$ are coprime.
Up to twisting with a line bundle, we can uniquely fix $\Lambda.f = -d$ with $0<d<r$.
Let us fix an almost $v$-fiberlike polarization $H$. Then the moduli space of semistable torsion free sheaves
with Mukai vector $v$ on $X$ is an irreducible smooth projective variety of dimension $2t$, where
$2-2t = \langle v,v \rangle$. We will denote this moduli space by $M_X(v)$.

Let $\widehat{X}$ be the relative Jacobian of $X$, that is the elliptic surface over $\PP^1$, whose fiber over
a point $c$ is the Jacobian of the elliptic curve $X_c$. The surface $\widehat{X}$ can be described as a
relative moduli space over $\PP^1$, for definition and properties we refer to \cite{huylehn}, in
particular Theorem 4.3.7. 
Let $\kp$ be the universal sheaf over $X \times \widehat{X}$. We consider the 
Fourier-Mukai transform $\Psi: \Db(X) \to \Db(\widehat{X})$ with kernel $\kp$, which is defined as
$$\Psi(-) = R\hat{p}_* (p^* (-) \otimes \kp),$$
where $p$ and $\hat{p}$ denote the projection from $X \times \widehat{X}$ to $X$ and $\widehat{X}$ respectively.
It is well known that $\Psi$ is an equivalence between $\Db(X)$ and $\Db(\widehat{X})$. We will denote by
$\Phi$ the relative quasi-inverse of $\Psi$, that is $\Phi \circ \Psi = [1]$. Denote by $\kq$ the kernel of $\Phi$.

Let $\hat{v} = (d,\hat{\Lambda},\hat{c})$ be the Mukai vector on $\widehat{X}$ with:
\begin{enumerate}
\item $\hat{\Lambda}.\hat{f} = r$,
\item $\langle\hat{v},\hat{v}\rangle = \langle v,v \rangle$.
\end{enumerate}
The moduli space $M_{\widehat{X}}(\hat{v})$ is irreducible
smooth projective of the same dimension $2t$ as $M_X(v)$.

In this Section we give the proof of the following Theorem.

\begin{theorem}\label{thm:iso-on-modspaces}
In the previous notations, if $t<r$, the moduli spaces $M_X(v)$ and $M_{\widehat{X}} (\hat{v})$ are isomorphic
and the isomorphism is induced by the Fourier-Mukai transform $\Psi$. 
\end{theorem}

Consider the $P$-stability corresponding to semistability with respect to $H$ described in
Theorem \ref{theorem:stable-is-Pstable}. Recall that it is given by three different conditions,
namely (h) and (tf), which ensures that a $P$-stable object is a torsion free sheaf
with given Hilbert polynomial, and the orthogonality condition (o).
We will show first that the condition (o) is well behaved with respect to the Fourier-Mukai transform. Then the image
under $\Psi$ of any stable sheaf satisfies an orthogonality condition. Then if a sheaf in $M_X(v)$ is mapped by $\Psi$
to a torsion free sheaf, this sheaf will be stable.

Let us start by recalling the general definition of ${\mathrm{WIT}}_i$ and ${\mathrm{IT}}_i$ sheaves.

\begin{definition}\label{def:wit}
Let $\Pi : \Db(Y) \to \Db(Z)$ be a functor between the derived categories of two smooth projective varieties.
A sheaf $E$ on $Y$ is $\Pi\WIT_i$ if $\Pi(E)[i]$ is a sheaf on $Z$, that is an object in $\Db(Z)$ whose only nontrivial
cohomology sheaf lives in degree zero. We denote in this case $\Pi(E)[i]$ by $\widehat{E}$. A $\Pi\WIT_i$ sheaf $E$ is 
$\Pi\IT_i$ if $\widehat{E}$ is a vector bundle.
\end{definition}

Remark that an object $E$ in $\Db(X)$ is $\Psi\WIT_i$ if and only if $\Psi(E)$ is $\Phi\WIT_{1-i}$, just
because $\Phi \circ \Psi = [1]$.

\subsection{From $M_X(v)$ to $M_{\widehat{X}}(\hat{v})$}

\begin{lemma}\label{lemma-recall-bridg}
Let $E$ be an object in $\Db(X)$ of rank $a$ and fiber degree $b$, then $\Psi(E)$ has rank $-b$ and fiber
degree $a$.

Let $E$ be a semistable torsion free sheaf on $X$ with negative slope. Then $E$ is $\Psi\WIT_1$.
\end{lemma}
\begin{proof}
These are special cases of \cite[Theorem 5.3 and Lemma 6.4]{bridgellip}, once remarked that the matrix associated to
$\Psi$ is
$$\left( \begin{array}{cc}
          0 & 1 \\
          -1 & 0
         \end{array}\right).$$
\end{proof}

For the convenience of the reader we include the following table which
shows what gives $\Psi$ applied to sheaves supported on fibers. To do so
we introduce the shorthand $T_{a,b}$ for a rank $a$ vector bundle of degree $b$ on a fiber
of $X \to C$ considered as a sheaf on $X$.


\medskip
\begin{center}
\begin{tabular}{ p{4.9em} p{10em} c p{8.6em} p{10.5em} }
\toprule
object               & description  & \hspace{2em} & $\Psi$(object)
&
description \\
\midrule
$T_{a,b}$, $b>0$     & stable bundle on a fiber   && $T_{b,-a}$
& stable
bundle on a fiber \\
$T_{a,b}$, $b<0$     & stable bundle on a fiber   && $T_{-b,a}[1]$
& stable
bundle on a fiber, shifted to degree 1 \\
$k(x)=T_{0,1}$       & skyscraper sheaf of $x$    && ${\cal
P}_x=T_{1,0}$ & line
bundle of degree 0 on a fiber \\
${\cal P}_x=T_{1,0}$ & line bundle of degree 0 on a fiber &&
                                                      $k(x)[-1]=
T_{0,1}[-1]$ &
skyscraper sheaf, shifted to
degree 1 \\
\bottomrule
\end{tabular}
\end{center}
\medskip

\begin{lemma}\label{ortho-goes-to-ortho}
For any $m$ and $n$ in $\ZZ$, an object $A$ in $\Db(X)$ is $(m,n)^{\perp}$ if and only if $\Psi(A)$ is $(-n,m)^{\perp}$.
\end{lemma}

\begin{proof}
Just a straightforward calculation using the matrix $S$ in the definition \ref{Def:orthog-condition}.
\end{proof}

\begin{remark}
Let $Y:= J_X(a,b)$ and consider the relative Fourier-Mukai equivalence $\Pi: \Db(X) \to \Db(Y)$
with the associated matrix $P$ as described by Bridgeland \cite{bridgellip}. Then the same proof
of Lemma \ref{ortho-goes-to-ortho} shows that $A$ in $\Db(X)$ is $(m,n)^{\perp}$ if and only if
$\Pi(A)$ is $((P \cdot {m \choose n})^t)^{\perp}$ in $\Db(Y)$.
\end{remark}

\begin{corollary}\label{image-of-stable-is-orth}
Let $E$ be a stable sheaf in $M_X(v)$. Then $E$ is $\Psi\WIT_1$ and the sheaf $\widehat{E}$ is $(d,r)^{\perp}$.
\end{corollary}
\begin{proof}
Recall that the slope of $E$ is $-d/r <0$, apply Lemmas \ref{lemma-recall-bridg} and \ref{ortho-goes-to-ortho}
to prove that $\Psi(E)$ is $(-d,-r)^{\perp}$ and conclude by recalling $\widehat{E}:=\Psi(E)[1]$.
\end{proof}

\begin{proposition}\label{E-hat-is-tf}
Let $E$ be a stable sheaf in $M_X(v)$. The sheaf $\widehat{E}$ is torsion free if and only if $E$
is a vector bundle such that for any fiber $F$ of $\pi$, the restriction $E_F$ does not admit a sub-bundle of positive slope.
\end{proposition}
\begin{proof}
Suppose first that $E$ is not locally free and consider the sequence
$$0 \longrightarrow E \longrightarrow E^{\vee \vee} \longrightarrow \tau \longrightarrow 0.$$
Since $\tau$ is zero-dimensional, it is $\Psi\WIT_0$.
Applying $\Psi$, we get $\Psi(\tau) \subset \Psi(E)[1] = \widehat{E}$, which shows that $\widehat{E}$ is not torsion free.

Now let $E$ be a vector bundle.
Suppose that the sheaf $\widehat{E}$ has a zero dimensional torsion, that is, there exists a point $p$ of $\widehat{X}$
such that $k(p) \subset \widehat{E}$, and let $\widehat{X}_0$ be the fiber of $\hat{\pi}$ containing $p$.
Then for any line bundle $L_0$ in $J(\widehat{X}_0)$, there is a nontrivial morphism $i_* L_0 \to \widehat{E}$,
where $i$ is the embedding of $\widehat{X}_0$ in $\widehat{X}$. Equivalently, for all points $x$ in $X_0$, there is a
nontrivial morphism $\Psi(k(x)) \to \widehat{E} = \Psi(E)[1]$. Since $\Psi$ is an equivalence, we get
that for all points $x$ in the fiber $X_0$, the extension group $\Ext^1(k(x),E)$ is nontrivial, which
contradicts the locally freeness of $E$.

If $T$ is a non-zero torsion subsheaf of $\widehat{E}$, it is then pure of dimension 1. Moreover, it has to
be supported on a finite number of fibers of $\hat{\pi}$. Indeed, if $\widehat{X}_0$ is a fiber such that $\widehat{X}_0 \cap
\supp{T}$ is a finite set of points, then the restriction of $E$ to $X_0$ is unstable,
and this can happen only for finitely many fibers.

If $\widehat{X}_0$ is a fiber in the support of $T$ and $T_0$ the restriction,
then $\mu_{max}(T_0) < 0$, where $\mu_{max}$ denotes the
maximal slope of a subsheaf of $T_0$
on the elliptic curve $\widehat{X}_0$.
Indeed, if $\mu_{max}(T_0) \geq 0$, we would have a nontrivial morphism $L_0 \to T_0$
for some $L_0$ in $J(\widehat{X}_0)$ and this would be a contradiction, as $E$ is a vector bundle, as before.
Suppose for simplicity that $T$ is pure supported on a single fiber $\widehat{X}_0$, then $T$ is $\Phi\WIT_1$. Let
$\check{T}:=\Phi(T)[1]$. The sheaf $\check{T}$ is supported on the fiber $X_0$ and has positive slope. Then:
$$\Ext^1_X(E,\check{T}) = \Ext^1_X(\check{T},E) = \Hom_{\widehat{X}}(T,\widehat{E}) \neq 0,$$
where we applied Serre duality at the first step. Since $E$ is locally free,
$$\Ext^1_X(E,\check{T}) = \Ext^1_{X_0}(E_0,\check{T}) = \Hom_{X_0}(\check{T},E_0),$$
where we applied Serre duality on $X_0$. Then we get the proof.
\end{proof}

\begin{lemma}\label{r^2-condition}
Let $E$ be a vector bundle on $X$ of Mukai vector $v$. If $t < r+d$, the restriction of $E$ to any fiber
of $\pi : X \to \PP^1$
does not admit a sub-bundle of positive slope.
\end{lemma}
\begin{proof}
Suppose $E$ is a vector bundle and $B$ a sheaf supported on a fiber $X_0$ with positive slope. As a sheaf on
$X_0$, $B$ has strictly positive rank and degree, say $r_0$ and $d_0$ respectively. Then part (1) of
Proposition \ref{ppp-fib} gives $rd_0 + r_0 d \leq t$. Since $r_0,d_0 \geq 1$, this cannot happen for $t<r+d$.
\end{proof}

\begin{lemma}\label{r-condition}
Let $E$ be a semistable sheaf on $X$ of Mukai vector $v$. If $t<r$, then $E$ is a vector bundle.
\end{lemma}
\begin{proof}
Suppose that $E$ has a unique singular point $x$. Consider the exact sequence
$$0 \longrightarrow E \longrightarrow E^{\vee \vee} \longrightarrow k(x) \longrightarrow 0.$$
The Mukai vector of $E^{\vee \vee}$ is $v' = (r,\Lambda,c+1)$ and $E^{\vee \vee}$ is semistable. The
proof follows by checking that for $t<r$ the moduli space $M_X(v')$ is empty. 
\end{proof}

\begin{corollary}\label{Psi-of-stable-is-tf}
If the Mukai vector $v$ is such that $t<r$, then for all $E$ in $M_X(v)$, the sheaf $\widehat{E}$ is torsion free.
\end{corollary}

Combining Corollaries \ref{image-of-stable-is-orth} and \ref{Psi-of-stable-is-tf} we get that, if $t < r$,
the sheaf $\widehat{E}$ is semistable of Mukai vector $\hat{v}$ for all $E$ in $M_X(v)$.

\subsection{From $M_{\widehat{X}}(\hat{v})$ to $M_X(v)$}

Consider the equivalence $\Phi:\Db(\widehat{X}) \to \Db(X)$, the relative quasi inverse of $\Psi$.
Recalling that since $\Psi \circ \Phi = [1]$, it is not difficult to deduce the behavior of $\Phi$
on sheaves supported on fibers by the table we established for $\Psi$ previously.

\begin{lemma}\label{lem-wit-is-it}
Let $E$ be a sheaf in $M_{\widehat{X}}(\hat{v})$. Then $E$ is $\Phi\WIT_0$ if and only if it is $\Phi\IT_0$.
\end{lemma}
\begin{proof}
Suppose $E$ is $\Phi\WIT_0$. The support of $\Phi(E)$ is $X$.
Let $T \hookrightarrow \Phi(E)$ be the maximal torsion subsheaf.
We have the exact sequence
\begin{equation}\label{tors-of-phie}
0 \longrightarrow T \longrightarrow \Phi(E) \longrightarrow F \longrightarrow 0.
\end{equation}
Suppose $T$ has zero dimensional torsion, that is, there exists a point $p$ of $X$
such that $k(p) \subset \Phi(E)$, and let $X_0$ be the fiber of $\pi$ containing $p$.
Then for any line bundle $L_0$ in $J(X_0)$, there is a nontrivial morphism $i_* L_0 \to \Phi(E)$,
where $i$ is the embedding of $X_0$ in $X$. Equivalently, for all point $x$ in $\widehat{X}_0$, there is a
unique nontrivial morphism $\Phi(k(x))[1] \to \Phi(E)$. We get then
that for all points $x$ in the fiber $\widehat{X}_0$, the extension group $\Ext^1(E,k(x))$ is nontrivial, which
contradicts the locally freeness of $E$. Then $T$ has to be pure of dimension 1. Moreover, $T$ has to be supported on
a finite number of fibers. Indeed, if the support of $T$ meets the fibers transversally, this would contradict the
orthogonality condition.

It follows that $T$ is a fiber sheaf. Suppose that $\mu_{max}(T) \leq 0$ on the fibers where $T$ is supported.
In this case, $T$ would be $\Psi\WIT_1$ and applying $\Psi$ we would find that $\Hom(\Psi(T),E) =
\Hom(T,\Phi(E)) \neq 0$, contradicting the torsion freeness of $E$. Indeed $\Psi(T)$ is a fiber sheaf as well.
It $\mu_{max}(T) > 0$ for some fiber, then $T$ would admit a sub-line bundle of non positive slope and the same
argument apply.

Then $\Phi(E)$ is torsion free and we have to prove that it is a vector bundle. Suppose this is not the case and consider
the exact sequence
$$0 \longrightarrow \Phi(E) \longrightarrow (\Phi(E))^{\vee \vee} \longrightarrow T' \longrightarrow 0$$
where $T'$ is zero dimensional. Since $T'$ is $\Psi\WIT_0$ and $\Psi(T')$ is a sheaf concentrated on a finite
number of fibers, applying $\Psi$ gives a contradiction to the torsion freeness of $E$.
\end{proof}

\begin{lemma}\label{phi-wit-little-t}
It $t<r$, every $E$ in $M_{\widehat{X}}(\hat{v})$ is $\Phi\WIT_0$.
\end{lemma}
\begin{proof}
Recall that by \cite[Lemma 6.1]{bridgellip}, for all sheaves $E$ on $\widehat{X}$, there exists a unique short exact sequence
\begin{equation}\label{seq-wit0-1}
0 \longrightarrow A \longrightarrow E \longrightarrow B \longrightarrow 0,
\end{equation}
where $A$ is $\Phi\WIT_0$ and $B$ is $\Phi\WIT_1$. 

Suppose that the rank of $B$ is positive. Then by the stability of $E$, we get that the slope of $B$ has to
be positive as well, but in this case $B$ would not be $\Phi\WIT_1$, by \cite[Lemma 6.2]{bridgellip}.
Then $B$ is a torsion $\Phi\WIT_1$ sheaf,
and then it is a fiber sheaf by \cite[Lemma 6.3]{bridgellip}. Suppose $B$ is supported in a single fiber $\widehat{X}_0$.
Then restricting the Fourier-Mukai transform to $\widehat{X}_0$, we have that $\mu_{max} (B) \leq 0$,
since $B$ is $\Phi\WIT_1$. Let $r_0$ and $-d_0$ be the rank and the degree respectively of $B$ seen as a sheaf
on $\widehat{X}_0$. Recall that $E$ has rank $d$ and fiber degree $r$ and apply point (1) of Proposition \ref{ppp-fib}
to get $r_0 r + d d_0 \leq t$. But since $r_0>0$ and $t<r$, this would be a contradiction.
\end{proof}

Putting together Lemmas \ref{lem-wit-is-it} and \ref{phi-wit-little-t},
we obtain that the image via $\Phi$ of a stable torsion free sheaf of Mukai vector $\hat{v}$ is a stable torsion
free sheaf (indeed, a vector bundle) of Mukai vector $v$.

Since $\Phi$ and $\Psi$ are each other quasi-inverse, up to the shift $[1]$ in $\Db(\widehat{X})$, we get the required
isomorphism.

\begin{remark}
Remark that, if we stick to the case $r=2$ described by Friedman \cite{friedellip}, Theorem \ref{thm:iso-on-modspaces}
gives an isomorphism between the moduli space and the Hilbert scheme only for $t=1$. The case $t=r=2$, in which Friedman
showed the existence of an isomorphism, will be described in the next Section.
\end{remark}

\subsection{The Euclid-Fourier-Mukai algorithm}
In the previous subsection we have shown how, under a condition on the dimension of the moduli space, the Fourier-Mukai
transform $\Psi$ induces an isomorphism between $M_X(v)$ and $M_{\widehat{X}}(\hat{v})$. The very special choice of $\widehat{X}$
could look more restrictive than the construction performed by Bridgeland in \cite{bridgellip}. We
show that it is indeed enough to consider, given an elliptic surface $X$, the Jacobian $\widehat{X}$ and perform a finite
number of steps to get a more general description.

Given coprime integers $r$ and $d$, with $0<d<r$,
we can consider the Euclidean algorithm applied to the pair $(r,-d)$: its
elementary step is given by sending $(r,-d)$ to $(d,-s)$, with $-s \equiv r$ modulo $d$ and $0<s<d$.
Since $r$ and $d$ are coprime, the algorithm will
end with the pair $(1,0)$.

The Mukai vector $v = (r, \Lambda,c)$ on $X$ is uniquely identified with the integers $r$, $-d$ and $t$,
where $-d$ is the fiber degree of $\Lambda$ and $2t$ the expected dimension of the moduli space.
Recall that $r$ and $d$ are coprime, $r \geq 2$ and $0<d<r$.
Consider the Jacobian $\widehat{X}$ of $X$ and the Fourier-Mukai equivalence $\Psi : \Db(X) \to \Db(\widehat{X})$. The functor
$\Psi[1]$ sends sheaves with a Mukai vector
$v$ on $X$ to sheaves with a Mukai vector $\hat{v}$ on $\widehat{X}$. The integer $t$ is invariant under this operation,
while the pair $(r,-d)$ is sent to the pair $(d,r)$. We can fix a line bundle $L$ on $\widehat{X}$ such that, for all
$E$ in $M_{\widehat{X}}(\hat{v})$, the sheaf $E \otimes L$ has fiber degree $-s$,  where $-s \equiv r$ modulo $d$ and $0<s<d$.
In particular,
$s$ and $d$ are coprime. This last operation gives rise to a natural isomorphism between $M_{\widehat{X}}(\hat{v})$ and
$M_{\widehat{X}}(\hat{w})$, where $\hat{w}$ is the Mukai vector of $E \otimes L$.

We can now define the Euclid-Fourier-Mukai algorithm as follows: given a Mukai vector $v=(r,\Lambda,c)$ as before,
consider the integers $t$ (half dimension of the moduli space), $r$ (rank) and $-d$ (fiber degree), with
$r$ and $-d$ coprime and $0<d<r$. These three integers uniquely identify the vector $v$. Consider the
Fourier-Mukai transform $\Psi$: it takes (up to a shift and a tensor by a line bundle) the pair $(r,-d)$ to
the pair $(d,-s)$, leaving $t$ invariant. If $d=1$ and $s=0$, we stop. If $0<s<d$, we can restart. The
Euclid-Fourier-Mukai algorithm is then obtained by applying this procedure to the moduli space $M_{\widehat{X}}(\hat{v})$.
Remark that the algorithm always ends with the Mukai vector $(1,0,t)$.

As a consequence, starting from the moduli space $M_X(v)$,
the Euclid-Fourier-Mukai algorithm ends
up with the Hilbert scheme $\Hilb^t(X)$
in a finite number of steps.
Indeed remember that, since $\pi:X \to\PP^1$ admits a section, $X$ is
isomorphic to its Jacobian $\widehat X$.

Since the rank is strictly decreasing and $t$ is constant, the condition given by
Theorem \ref{thm:iso-on-modspaces} gets more strict at every step. Consider the Mukai vector $v$, with $t$ and $(r,-d)$ as before
and the Euclid-Fourier-Mukai algorithm. The last step has the form $(l,-1) \to (1,0)$ for $l \geq 2$ integer.

\begin{corollary}\label{cor:iso-M-H}
If $t < l$, the moduli space $M_X(v)$ is isomorphic to $\Hilb^t(X)$.
\end{corollary}

Given the pair $(r,-d)$, there exist integers $a$ and $b$ such that $br+ad=1$, and we can fix $a$ uniquely such that
$0<a<r$. We have the following result by Bridgeland.

\begin{theorem}\label{iso-bridge}
If $t < r/a$, the moduli space $M_X(v)$ is isomorphic to $\Hilb^t(X)$.
\end{theorem}
\begin{proof}
\cite[Lemma 7.4]{bridgellip}
\end{proof}

Corollary \ref{cor:iso-M-H} and Theorem \ref{iso-bridge} give indeed the same result.

\begin{lemma}\label{us=bridg}
With the previous notations, $l = \lceil r/a \rceil$.
\end{lemma}
We will give the proof of this rather technical Lemma in the Appendix, see Lemma \ref{elem-nt1}.

\section{Birational transformations between the moduli spaces}

In the previous Section, we have shown that if $t<r$, we can
establish an isomorphism between $M_X(v)$ and $M_{\widehat{X}}(\hat{v})$. In this Section, we deal
with $t \geq r$ and we describe the birational correspondence between the moduli spaces, obtaining
a Mukai flop for $r \leq t < r+d$, which is an isomorphism if and only if $r=t=2$. It is also clear that,
for any $t \geq r > 2$, the map cannot be an isomorphism.
For $r=2$, we argue how in the case $t=4$, the conjecture formulated by Friedman should be false.
Let us state the main Theorem of this Section.
\begin{theorem}\label{mukai-flop}
Let $r \leq t < r+d$. Then the Fourier-Mukai equivalence $\Psi$ induces a Mukai flop between
$M_X(v)$ and $M_{\widehat{X}}(\hat{v})$, unless $t=r=2$, in which case it induces an isomorphism.
\end{theorem}

\subsection{New compactifications for the moduli space of vector bundles}\label{new-moduli}
Before giving the proof of Theorem \ref{mukai-flop}, let us describe a very interesting implication.
As shown in the previous
Section, the generic stable sheaf $E$ of $M_X(v)$ is sent by $\Psi$ to a stable sheaf $\widehat{E}$. Indeed,
$\widehat{E}$ always satisfies a suitable orthogonality condition. The stability follows from torsion
freeness of $\widehat{E}$, which is assured for $E$ generic by Proposition \ref{E-hat-is-tf}. The indeterminacy
locus of the birational map is then given by those $E$ in $M_X(v)$ for which $\widehat{E}$ is not torsion free.
On the other hand, for the generic
stable sheaf $E$ in $M_{\widehat{X}}(\hat{v})$, $\Phi(E)$ is
a stable sheaf in $M_X(v)$, which is indeed a stable vector bundle. The indeterminacy locus of the
birational map induced by $\Phi$ is then given by those $E$ for which $\Phi(E)$ is not a vector bundle. We
are showing that in this case $\Phi(E)$ is a complex in $\Db(X)$ with two non-vanishing cohomologies.

Recall we are considering $M_X(v)$ as moduli space of $P$-stable objects in $\Db(X)$ and that $\Psi$
induces an isomorphism between $M_X(v)$ and the moduli space of $\Psi(P)$-stable objects in $\Db(\widehat{X})$,
which, as a set, is nothing but $\Psi (M_X(v))$. We are then giving an irreducible
projective smooth fine moduli space $N_{\widehat{X}}$ for
$\Psi(P)$-stable sheaves over $\widehat{X}$. The generic element of $N_{\widehat{X}}$ is a stable vector bundle of Mukai vector
$\hat{v}$. For $t\geq r$, $N_{\widehat{X}}$ contains also sheaves with torsion and is not isomorphic to $M_{\widehat{X}}(\hat{v})$.
This provides a new compactification of the moduli space of vector bundles via sheaves with torsion.

On the other side, we can consider $M_{\widehat{X}}(\hat{v})$ as moduli space of $P$-stable objects in $\Db(\widehat{X})$ and
$\Phi$ induces an isomorphism between $M_{\widehat{X}}(\hat{v})$ and the moduli space of $\Phi(P)$-stable objects in
$\Db(X)$, which, as a set, is nothing but $\Phi(M_{\widehat{X}}(\hat{v}))$. We are then giving an irreducible projective smooth fine
moduli space $N_X$ for $\Phi(P)$-stable objects over $X$. The generic element of $N_X$ is a stable vector bundle
of Mukai vector $v$. For $t\geq r$, $N_X$ contains also true complexes and
is not isomorphic to $M_X(v)$. This provides a new compactification of the moduli space of vector bundles via complexes. 

\subsection{Proof of Theorem \ref{mukai-flop}}
Let us denote the birational map induced by $\Psi$ as 
$$\varepsilon: M_X(v) \dashrightarrow M_{\widehat{X}}(\hat{v}).$$
Let $r \leq t < r+d$. If $E$ is a stable sheaf in $M_X(v)$, Corollary \ref{image-of-stable-is-orth} ensures that
the sheaf $\widehat{E}$ is $(d,r)^{\perp}$. Then $\widehat{E}$ is not in $M_{\widehat{X}}(\hat{v})$ if and only if it is not
torsion free.
\begin{lemma}\label{E-hat-not-tf}
Let $E$ be a stable sheaf in $M_X(v)$ and $t<r+d$. Then $E$ is $\Psi\WIT_1$. The sheaf $\widehat{E}$ is not torsion free
if and only if $E$ is not a
vector bundle. Moreover this is the case if and only if there is a point
$x$ in $X$ which appears in the short exact sequence
$$0 \longrightarrow E \longrightarrow E^{\vee \vee} \longrightarrow k(x) \longrightarrow 0.$$
\end{lemma}
\begin{proof}
The first equivalence follows from Proposition \ref{E-hat-is-tf}, Lemma \ref{r^2-condition}
and proof of Corollary \ref{Psi-of-stable-is-tf}. For the second one, suppose that there are exactly two points
$x_1$ and $x_2$ where $E$ fails to be locally free. Then we have the exact sequence
$$0 \longrightarrow E \longrightarrow E^{\vee \vee} \longrightarrow k(x_1) \oplus k(x_2) \longrightarrow 0,$$
and we get the proof by arguing as in the proof of Lemma \ref{r-condition}, since $t < r+d < 2r$.
\end{proof}
Let $E$ be a non locally free sheaf in $M_X(v)$ and denote by $u$ the Mukai vector of $E^{\vee \vee}$.
We have $u=(r,\Lambda,c+1)$, independently of the point $x$.
Remark that $E^{\vee \vee}$ is a stable vector bundle
and the moduli space $M_X(u)$ has dimension $2t-2r$ and is then isomorphic to $M_{\widehat{X}}(\hat{u})$.
\begin{proposition}\label{indet-of-psi}
The indeterminacy locus of $\varepsilon$ is a $\PP^{r-1}$-bundle over
$X
\times M_X(u)$ and naturally isomorphic via $\Psi$ to
the a $\PP^{r-1}$-bundle over $X \times M_{\widehat X} (\hat{u})$ associated
to the universal extension of the form $\PP\Ext^1(F,\kp_x)$, for $F$ in $M_{\widehat{X}}(\hat{u})$ and $\kp_x$ in $\widehat{\widehat{X}}
\cong X$. In particular, the indeterminacy locus has codimension $r-1$.
\end{proposition}
\begin{proof}
Let $E$ in $M_X(v)$ be such that $\Psi(E)$ is not in $M_{\widehat{X}}(\hat{v})$. Then, by Lemma \ref{E-hat-not-tf},
there is a point $x$ in $X$ and an exact sequence
\begin{equation}\label{seq-e-doub-dual}
0 \longrightarrow E \longrightarrow E^{\vee \vee} \longrightarrow k(x) \longrightarrow 0.
\end{equation}
Moreover, any element of $M_X(u)$ appears in a sequence
like (\ref{seq-e-doub-dual}) for some $E$ in $M_X(v)$ and some point $x$ in $X$. Apply $\Psi$ to get
the exact sequence
$$0 \longrightarrow \kp_x \longrightarrow \widehat{E} \longrightarrow \widehat{E^{\vee \vee}} \longrightarrow 0$$
of sheaves on $\widehat{X}$. The sheaves $\widehat{E}$ are then parameterized, as $E$ moves in the indeterminacy locus, by extensions
of an element of $M_{\widehat{X}}(\hat{u})$ by a degree zero line bundle on a fiber of $\hat{\pi}$. Conversely, it is clear that
any such extension corresponds to the image of an element of $M_X(v)$ whose image under $\Psi$ is not torsion free.

The dimension of the fiber is easily calculated by $\Ext^1(F,\kp_x) \cong \Hom(E^{\vee\vee},k(x))$.
Having computed the dimension of $M_X(u)$, the statement about the
codimension is straightforward.
\end{proof}

Consider the equivalence $\Phi$, the relative quasi inverse of $\Psi$. Then $\Phi$ induces a birational
map
$$\rho: M_{\widehat{X}}(\hat{v}) \dashrightarrow M_X(v).$$
The same arguments used to prove Lemma \ref{lem-wit-is-it} apply to give the following result.

\begin{lemma}
Let $E$ be a sheaf in $M_{\widehat{X}}(\hat{v})$. Then $E$ is $\Phi\WIT_0$ if and only if it is $\Phi\IT_0$.
\end{lemma}

\begin{proposition}\label{indet-of-phi}
The indeterminacy locus of $\rho$ is the projective bundle over $X \times M_{\widehat X} (\hat{u})$ associated
to the universal extension of the form $\PP\Ext^1(\kp_x,F)$, for $F$ in $M_{\widehat{X}}(\hat{u})$ and $\kp_x$ in $\widehat{\widehat{X}}
\cong X$.
\end{proposition}

\begin{proof}
Let $E$ be in $M_{\widehat{X}}(\hat{v})$.
First, $E$ is not in the
indeterminacy locus of $\rho$ if and only if $E$ is $\Phi\WIT_0$. Indeed, if $E$ is $\Phi\WIT_0$,
then it is $\Phi\IT_0$ and $\Phi(E)$ is then a locally free sheaf on $X$ of rank $r$, fiber degree $-d$ and
it is $(r,-d)^{\perp}$. It follows that $\Phi(E)$ is in $M_X(v)$.
Conversely, the fact that $\rho(E)$ is in $M_X(v)$ tells in particular that $\Phi(E)$ is a sheaf.

Recall that by \cite[Lemma 6.1]{bridgellip}, for all sheaves $E$ on $\widehat{X}$, there exists a unique short exact sequence
\begin{equation}
0 \longrightarrow A \longrightarrow E \longrightarrow B \longrightarrow 0,
\end{equation}
where $A$ is $\Phi\WIT_0$ and $B$ is $\Phi\WIT_1$. 

As in the proof of Lemma \ref{phi-wit-little-t}, we can show that $B$ is a pure fiber sheaf of non-positive
maximal slope. Suppose that $B$ is supported on a single fiber and let $r_0$ and $-d_0$ be respectively
its rank and its degree as a sheaf on $\widehat{X}_0$. Recall that $E$ has rank $d$ and fiber degree
$r$. Then applying point (1) of Proposition \ref{ppp-fib}, we get $r_0r+dd_0 \leq t$. Since $r_0 >0$,
we get that $B$ has rank 1 and degree 0. Moreover, $B$ cannot be supported on any other fiber, by point (2)
of Proposition \ref{ppp-fib}.
This identifies also the Mukai vector of $A$, up to a twist with a line bundle.

Finally, any extension of the form
$$0 \longrightarrow F \longrightarrow E \longrightarrow \kp_x \longrightarrow 0$$
is stable. Indeed, such an $E$ is torsion free and satisfies the orthogonality condition.
\end{proof}

Now let $U \subset M_X(v)$ be the subset where $\varepsilon$ is defined and $\widehat{U} \subset M_{\widehat{X}}(\hat{v})$
the subset where $\rho$ is defined. It is now clear that $\varepsilon(U) = \widehat{U}$ and $\rho(\widehat{U}) = U$.
Moreover, $\varepsilon \circ \rho = \rho \circ \varepsilon = {\mathrm{Id}}$.
Compare Propositions \ref{indet-of-psi} and \ref{indet-of-phi} by using Serre duality.
A dimension calculation shows that the indeterminacy loci have codimension $r-1$.
If $t=r=2$ the indeterminacy loci have codimension 1, which gives an isomorphism. In any other case, since $r>2$,
we are dualizing a projective bundle and then we have a nontrivial Mukai flop.

\begin{remark}
Suppose that $t \geq r+d$. Then the indeterminacy locus of $\varepsilon$ gets bigger, since it contains all
those stable vector bundles $E$ which have ``very unstable'' restriction to some fiber. Anyway, if $r>2$, it is
clear by the description of the indeterminacy loci of $\varepsilon$ and $\rho$ that we would
never have an isomorphism for $t \geq r$. For $r>2$, we obtained that the condition $t<r$ is optimal.
\end{remark}

\begin{remark}
If $r=2$, consider the case $t=4$. In this case, there are sheaves $E$ in $M_X(v)$ such that
the cokernel of the map $E \hookrightarrow E^{\vee \vee}$ has length 2. These sheaves belong to
the indeterminacy locus and their locus forms a projective bundle $\PP \to X \times M_{\widehat{X}}(\hat{w})$,
where $w=(2,\Lambda,c+2)$. If we go through the proof of Theorem \ref{mukai-flop}, we can define
a stratum in the indeterminacy locus of $\rho$ which is parameterized by $\PP^{\vee}$. This stratum is
indeed given by those semistable sheaves admitting $\kp_{x_1} \oplus \kp_{x_2}$ as a quotient. 
This should lead to a non-trivial flop and then to a contradiction to \cite[III, Conjecture 4.13]{friedellip}.
\end{remark}

\section{Isomorphisms between the Picard groups}\label{sect-pic-iso}

The birational map $\ee$ has a very nice description if the dimension of the moduli space is not too
big. It is clear, that as soon as $t \geq r+d$, there are stable sheaves $E$ in $M_X(v)$ which
admit on some fiber a destabilizing subbundle of positive slope. This makes the indeterminacy
locus of $\ee$ more complicated. Anyway, in this Section we show that,
with no condition on the dimension $2t$, the map $\ee$ induces
an isomorphism between the Picard groups $\Pic(M_X(v))$ and $\Pic(M_{\widehat{X}}(\hat{v}))$. 

\begin{theorem}\label{iso-betw-pic}
The birational map $\varepsilon: M_X(v) \dashrightarrow M_{\widehat{X}}(\hat{v})$ induces an isomorphism
$$\Pic(M_X(v)) \simeq \Pic(M_{\widehat{X}}(\hat{v})).$$
\end{theorem}

\begin{corollary}\label{iso-pic-of-hilb}
Let $v = (r, \Lambda, c)$ be a Mukai vector on $X$ with $\Lambda.f = -d$, with $r$, $d$ coprime, and $d$
any integer. Then we have an isomorphism
$$\Pic(M_X(v)) \simeq \Pic(\Hilb^t(X)),$$
where $2t$ is the dimension of the moduli space $M_X(v)$.
\end{corollary}
\begin{proof}
By tensoring with a line bundle, we can suppose, up to isomorphism of $M_X(v)$, that $0<d<r$.
The proof follows applying recursively Theorem \ref{iso-betw-pic} to the steps of the Euclid--Fourier--Mukai
algorithm.
\end{proof}

The whole Section is dedicated to the proof of Theorem \ref{iso-betw-pic}.

First of all, Proposition \ref{E-hat-is-tf} gives us the indeterminacy locus of the map $\varepsilon$
induced by $\Psi$. This is the union of
two closed subsets, $K$ and $L$, where:
$$K = \{ E \in M_X(v) \, \vert \, E {\text{ is not locally free}} \}$$
$$L = \{ E \in M_X(v) \, \vert \, \exists {\text{ a fiber }} X_0 {\text{ and a sheaf }} F_0 {\text{ on }} X_0
{\text{ with }} \mu(F_0) > 0 {\text{ and }} F_0 \hookrightarrow E_{\vert X_0} \}.$$  

On the other side, Lemma \ref{lem-wit-is-it} implies that the indeterminacy locus of the map $\rho$ induced by $\Phi$
is given by those stable sheaves $F$ in $M_{\widehat{X}}(\hat{v})$ which are not $\Phi\WIT_0$. Recall that any sheaf
$F$ on $\widehat{X}$ fits a unique exact sequence
\begin{equation}\label{sequence}
0 \longrightarrow A \longrightarrow F \longrightarrow B \longrightarrow 0,
\end{equation}
where $A$ is $\Phi\WIT_0$ and $B$ is $\Phi\WIT_1$. If $F$ is in $M_{\widehat{X}}(\hat{v})$ is then clear that
$B$ is a fiber sheaf of non-positive slope. Then the indeterminacy locus of $\rho$ can be seen as the
union of two closed subset $\widehat{K}$ and $\widehat{L}$, where:
$$\widehat{K} = \{ F \in M_{\widehat{X}}(\hat{v}) \, \vert \, \mu(B) = 0 \}$$
$$\widehat{L} = \{ F \in M_{\widehat{X}}(\hat{v}) \, \vert \, \mu(B) < 0 \}.$$

By Proposition \ref{indet-of-psi} We have that the codimension of $K$ in $M_X(v)$ is $r-1$.
Let us consider the stratification of $K$ whose strata are given by
$$K_i := \{ E \text{ in } K \, \vert \, l(\coker (E \to E^{\vee\vee})) \geq i \}.$$
It is clear that $K_1=K$, that $K_i$ is empty for $i \gg 0$, and that $K_{i+1}$ is a proper closed subset
of $K_i$ for all $i$. On the other side, any $F$ in $\widehat{K}$ fits the exact sequence
(\ref{sequence}) with $\deg B = 0$. We consider the stratification of $\widehat{K}$ defined by
$$\widehat{K}_i := \{ F \text{ in } \widehat{K} \, \vert \,
\mbox{fiber-}\rk(B) \geq i \}.$$
It is clear that $\widehat{K}_1=\widehat{K}$, that $\widehat{K}_i$ is empty for $i \gg 0$, and that
$\widehat{K}_{i+1}$ is a proper closed subset of $\widehat{K}_i$ for all $i$.
The codimension of $\widehat{K}$ is then established via duality between the strata.  

\begin{lemma}\label{codim-of-K^}
The subsets $K$ and $\widehat{K}$ have codimension $r-1$ in  $M_X(v)$
and $M_{\widehat{X}}(\hat{v})$ respectively.
\end{lemma}
\begin{proof}
We are actually proving more, that is that the strata $K_i$ and $\widehat{K}_i$ are locally given
by dual projective bundles over the same base.

Consider $F$ in $\widehat{K}_i$ and the unique short exact sequence (\ref{sequence}).
Let $\hat{w}_i$ be the Mukai vector of $A$, then the stratum $\widehat{K}_i$ is given,
outside $\widehat{K}_{i+1}$, by the projective bundle $\PP {\mathrm{Ext}}^1(B,A)$ over
$J_{\widehat{X}}(i,0) \times M_{\widehat{X}}(\hat{w}_i)$.
Dualizing the sequence and applying the functor
$\Phi$, arguing as in Proposition \ref{indet-of-phi} we get unique stable sheaf $E$ in $K_i$.
Moreover, $E$ belongs to $K_{i+1}$ if and only if $F$ belongs to $\widehat{K}_{i+1}$.
This shows, recalling that ${\mathrm{Ext}}^1(B,A) = {\mathrm{Ext}}^1(A,B)^*$ by Serre duality, that $K_i$ and
$\widehat{K}_i$ are locally given by dual projective bundles over the same base.
\end{proof}

\begin{lemma}\label{codim-of-L} 
The subset $L$ has codimension at least $2$ in $M_X(v)$.
\end{lemma}
\begin{proof}
Let $X_0$ be a fiber of $X \to \PP^1$, and $G_0\dual$ be a stable vector
bundle on the fiber $X_0$ of positive slope. Furthermore let $m \geq 0$
be an integer.
We denote by $L^{G_0,m}(v)$ the subset of vector bundles $[E] \in
M_X(v)$ such that $\mu_{\rm max}(E_{\vert X_0})=\mu(G_0\dual)$ and
$\hom(G_0\dual,E_{\vert X_0}) =m$. 
We consider the number $z=z_G=\rk(G)\rk(E_0)(\mu(G) - \mu(E_0)) =
r\deg(G)+d\rk(G) \geq r+d$. We claim that for $m>0$ the codimension
of $L^{G,m}(v)$ is at least $m(z_G+1)$, for all $t$. We prove this inductively over
the dimension $2t$ of the moduli space. The initial step is Proposition
\ref{ppp-fib} which shows that for $2t = \dim (M_X(v)) \leq 4$ these set are
empty.

We consider the the set $S$ of short exact sequences 
\[ 0 \to (E')\dual \to E\dual \to i_* G_0 \to 0 \]
with $[E] \in L^{G_0,m}(v)$ and $i:X_0 \to X$ is the closed embedding of the fiber.
Since $\dim \Hom(E\dual ,i_* G_0) = \dim\Hom(G_0\dual,E_{\vert X_0}) =m$, and any
non-zero $\pi: E\dual \to G_0$ is surjective, the set $S$ is a
$\PP^{m-1}$-bundle over $L^{G_0,m}(v)$.
Let $m'$ be the dimension of $\Hom((E')\dual,G_0)$.
Denoting the kernel of $(E\dual)_{\vert X_0} \to G_0$ by $Q_0$ we obtain two short exact sequences:
\[
0 \to Q_0 \to (E\dual)_{\vert X_0} \to G_0 \to 0 
\qquad
\mbox{and}
\qquad
0 \to G_0 \to ((E')\dual)_{\vert X_0} \to Q_0 \to 0
\]
on $X_0$.
Using $\hom(G_0,G_0)=\hom^1(G_0,G_0)=1$,
we compute from the first sequence that $\hom(Q_0,G_0) = m-\ee$,
and from the second that $m'=\hom(Q_0,G_0)+\ee'$ where $\ee,\ee' \in \{0,1\}$. We conclude that
$m' \in \{ m-1,m,m+1 \} $. We decompose $S = S^{m-1} \cup S^m \cup
S^{m+1}$ depending on the value of $m'$.
The sheaf $E'$ is again stable, since its restriction to a general fiber is stable.
As computed in Proposition \ref{ppp-fib} the sheaves $E'$ are parameterized by a moduli space
$M_X(v')$ of dimension $\dim M_X(v') = \dim M_X(v)-2z_G$.
Since $\Ext^1(i_* G_0, (E')\dual)$ is isomorphic to
$\Hom(G_0,(E')\dual_{\vert X_0})$ we deduce that the set $S^{m'}$ is
contained in a $\PP^{z_G+m'-1}$-bundle over the closed subset $L^{G_0,m'}(v')$
of $M_X(v')$. The
induction hypothesis gives us estimates for the dimension of
$L^{G_0,m'}(v')$. Now we have
\[ \dim(S^{m'}) \leq 2 \dim(M_X(v)) - 2z_G -m'(z_G+1) + z_G+m'-1
\]
Since $S$ is a $\PP^{m-1}$- bundle over 
$M_X^{G,m}(v)$ and $m \leq m'+1$, we conclude for the dimension 
\[ \dim(L^{G,m}(v)) \leq  \dim(M_X(v)) -m(z_G+1) -z_G+(m-m')z_G  \leq
\dim(M_X(v)) -m(z_G+1) \, .  \]
Since $z_G \geq r+d \geq 3$, we conclude the the codimension of
$L^{G,m}(v)$ is at least four.
Each sheaf $[E] \in L$ belongs to $L^{G,m}(v)$ for a fiber $X_0$ and a
stable $G_0$ on $X_0$. The sheaf $G_0$ is stable and its slope runs in a finite set of rational numbers, the
dimension $m$ of the morphism space is also bounded. The
choices of the fiber $X_0$ and of a stable $G_0$ on the fiber are both of dimension one. We finally get the proof.
\end{proof}

\begin{lemma}\label{codim-of-L^}
The subset $\widehat{L}$ has codimension at least 2 in $M_{\widehat{X}}(\hat{v})$.
\end{lemma}
\begin{proof}
We show that the codimension of $\widehat{L}$ in
$M_{\widehat{X}}(\hat{v})$ coincides with the codimension of $L$ in
$M_X(v)$. For any $F$ in $\widehat{L}$  we have a surjection to a stable
sheaf $B_{\min}$ of minimal possible degree concentrated on a fiber. In
particular $\mu(B_{\min})<0$. Therefore $B=\Psi(i_*G)$ for a stable
sheaf on the same fiber of positive degree. Denoting the kernel of the
surjection $F \to \Psi(i_*G)$ by $F'$, we see that $F$ corresponds to a
class in $\Ext^1(\Psi(i_*G), F')$. These extensions have the same
dimension as those in $\Ext^1(F',\Psi(i_*G))$ by Serre duality.
The latter correspond to objects in $L$.
\end{proof}

\begin{remark}
Analogously to Lemma \ref{codim-of-K^} one could prove that the subsets
$L$ and $\widehat{L}$ are locally described by dual projective bundles
over the same base.
\end{remark}

We are now left with the case $r=2$, in which $K$ and $\widehat{K}$ are divisors in $M_X(v)$ and 
$M_{\widehat{X}}(\hat{v})$ respectively. Anyway, if we consider an element $E$ of $K_1 \setminus K_2$,
there are a point $x$ of $X$ and an exact sequence
$$0 \longrightarrow E \longrightarrow E^{\vee\vee} \longrightarrow k(x) \longrightarrow 0.$$
Let $p$ be the map on $M_X(v)$
defined by the global sections of the generalized $\Theta$-line bundle
$\kl$ on $M_X(v)$, as in \cite{Hdual}, and let $\bar{M}$ be the image
via $p$ of $M_X(v) \setminus (K_2 \cup L)$. By \cite[Theorem 3.3]{Hdual}, $p$ is an isomorphism over
$(M_X(v) \setminus (K \cup L)) $ and a $\PP^1$-bundle over the image of
$K_1 \setminus K_2$ where all the kernels of surjections $\tilde E \to
k(x)$ are identified for fixed $\tilde E$ and $x \in X(k)$.

Likewise we obtain a map
$\hat p: M_{\widehat X}(\hat v) \setminus (\widehat K_2 \cup \widehat L)
\to \bar M$ by assigning to $F$ the divisor of jumping lines of $\Phi(F)$. 
We have $p(K_1)= \hat p(\widehat K_1)$ and both morphisms $K_1 \to p(K_1)$
and $\hat K_1 \to \hat p(\widehat K_1)$ are $\PP^1$-bundles outside $K_2$ and $\widehat{K}_2$ respectively.
Since $K_2$ and $L$ are of codimension at least two we conclude
$\Pic(M_X(v)) = \ZZ \cdot K_1 \oplus p^* \Pic(\bar M)$. The same way we
conclude $\Pic(M_{\hat X}(\hat v)) = \ZZ \cdot \widehat K_1 \oplus \hat p^* \Pic(\bar M)$.
This ends the proof of Theorem \ref{iso-betw-pic}.

\appendix
\section{Continued fractions and the Euclidean algorithm}
In this Appendix, we give a proof of Lemma \ref{us=bridg}.
We start with a version of the Euclidean algorithm for integers $(r,d)$
with $0<d<r$. We start by setting $(r_0,d_0)=(r,d)$ and proceed
inductively with
$(r_{i+1},d_{i+1})= (d_i, \lceil \frac{r_i}{d_i} \rceil d_i -r_i)$.
The algorithm stops when $d_{n+1}=0$, the corresponding $r_{n+1}$ is the
greatest common divisor of $r$ and $d$. We assume from now on
that $r$ and $d$ are
coprime. We denote the integers $\lceil \frac{r_n}{d_n} \rceil$ by $a_n$
and remark that $a_i \geq 2$. We see that we can recover $r$ and $d$
from the numbers $\{a_i\}_{i=0,\ldots,n}$ since
\[ \frac{r}{d} = \lceil a_0,a_1, \ldots, a_n \rceil =
a_0-\frac{1}{a_1-\frac{1}{a_2-\frac{1}{a_3-\frac{1}{\dots -\frac{1}{a_n}}}}} \, .
\]
Since $(r,d)=1$ there are integers $b$ and $c$ such that $rc+db=1$. The pair $(b,c)$
is uniquely determined when we assume
$0 \leq b < r$. Note that our assumption $0<d<r$ implies $0<b<r$.
Lemma \ref{us=bridg} is now
proved by point (4) of the following result.

\begin{lemma}\label{elem-nt1}
Let $\{a_i\}_{i=0,\ldots,n}$ be integers with $a_i \geq 2$. Let
$\frac{r}{d}= \lceil a_0,a_1, \ldots, a_n \rceil $ written
in reduced form. Let $(b,c)$ be integers with $rc+db=1$ and
$0  \leq b < r$. The following holds:
\begin{enumerate}
\item The number $\frac{r}{d}$ satisfies $\frac{r}{d}>1$.
\item The integer $b$ is positive.
\item If we reverse the order in the continued fraction for
$\frac{r}{d}$, then we
obtain the continued fraction expression of $\frac{r}{b}$.
In short:
$\frac{r}{b} = \lceil a_n,a_{n-1}, \ldots, a_0 \rceil $.
\item We have an equality $a_n= \lceil \frac{r}{b} \rceil$.
\end{enumerate}
\end{lemma}

\begin{proof}
(1) Follows by induction since $\lceil a_n \rceil =a_n \geq 2$, and
$\lceil a_{i-1},a_i, \ldots, a_n \rceil 
= a_{i-1}-\frac{1}{\lceil a_i, \ldots, a_n \rceil }$.\\
(2) Since $\frac{r}{d}>1$ we have $0<d<r$. Thus there are no solutions
$(b,c)$ of the equation $rc+db=1$ with $b=0$.\\
(3) We observe that $\left( \begin{array}{r} r \\d \end{array}
\right)$ is the first column of the SL$_2(\ZZ)$-matrix
\[ A = T^{a_0} \cdot S \cdot T^{a_1} \cdot S \cdot \dots T^{a_n} \cdot
S \qquad
\mbox{with} \quad
T=\left( \begin{array}{rr} 1&1 \\ 0&1\end{array} \right)
\mbox{, and }
S=\left( \begin{array}{rr} 0&-1 \\ 1&0\end{array} \right)
\,.\]
We write $A=\left( \begin{array}{rr} r&-b' \\ d&c'\end{array} \right)$
and set $R=\left( \begin{array}{rr}-1&0 \\ 0&1\end{array} \right)$.
We compute that
$R \cdot A^t \cdot R= \left( \begin{array}{rr} r&-d \\
b'&c'\end{array} \right)$.
On the other hand using the above product presentation for $A$ 
together with the two matrix equations
$R=R^{-1}$ and $R \cdot S^t \cdot (T^{a_i})^t \cdot R = T^{a_i} \cdot
S$ we obtain that:
\[ \begin{array}{rcll}
R \cdot A^t \cdot R
& = & R \cdot S^t \cdot (T^{a_n})^t \cdot S^t \cdot
(T^{a_{n-1}})^t \cdot \dots \cdot S^t \cdot (T^{a_0})^t \cdot R\\
& = &( R \cdot S^t \cdot (T^{a_n})^t \cdot R ) (R \cdot S^t \cdot
(T^{a_{n-1}})^t \cdot R) \cdot \dots \cdot (R \cdot S^t \cdot
(T^{a_0})^t \cdot R)\\
&=&T^{a_n} \cdot S \cdot T^{a_{n-1}} \cdot S \cdot \dots T^{a_0} \cdot S
\end{array}\\\]
This implies that $\frac{r}{b'}=\lceil a_n,a_{n-1}, \ldots, a_0 \rceil$.
We deduce from $a_i \geq 2$ and part (1) that $0< b' <r$. Since the
matrix $A$ is in SL$_2(\ZZ)$ we have $rc'+db'=1$. Thus $b=b'$.\\
(4) Is a direct consequence of the expression $\frac{r}{b}=\lceil
a_n,a_{n-1}, \ldots, a_0 \rceil$ derived in (3).
\end{proof}

\end{document}